\documentclass[11pt]{article}
\usepackage{graphicx}
\usepackage{bm}
\usepackage{amsmath,amsthm,amssymb,url}
\usepackage{multirow,bigdelim}
\usepackage{amscd}

\newtheorem{thm}{Theorem}[section]
\newtheorem{theorem}[thm]{Theorem}
\newtheorem{proposition}[thm]{Proposition}
\newtheorem{lemma}[thm]{Lemma}

\newtheorem{claim}[thm]{Claim}
\newtheorem{corollary}[thm]{Corollary}

\theoremstyle{definition}

\newcommand{\bp}{{\bm p}}
\newcommand{\bq}{{\bm q}}
\newcommand{\br}{{\bm r}}
\newcommand{\bmm}{{\bm m}}

\begin{document}
\title{Sufficient Conditions for the Global Rigidity of Graphs}
\author{Shin-ichi Tanigawa\thanks{
  Research Institute for Mathematical Sciences, Kyoto University, Kyoto 606-8502 Japan, +81-75-753-7793, \texttt{tanigawa@kurims.kyoto-u.ac.jp} 
  Supported by JSPS Grant-in-Aid for Scientific Research(A)(25240004) and JSPS Grant-in-Aid for Young Scientist (B), 24740058.}}
%\date{\today}
\maketitle
\begin{abstract}
We investigate how to find generic and globally rigid realizations of graphs in $\mathbb{R}^d$
based on elementary geometric observations.
Our arguments lead to new proofs of a combinatorial characterization of 
the global rigidity of graphs in $\mathbb{R}^2$ by Jackson and Jord{\'a}n
and that of body-bar graphs in $\mathbb{R}^d$ recently shown by Connelly, Jord{\'a}n, and Whiteley.
We also extend  the 1-extension theorem and Connelly's composition theorem,
which are main tools for generating globally rigid graphs in $\mathbb{R}^d$. 
In particular we show that any vertex-redundantly rigid graph in $\mathbb{R}^d$ is 
globally rigid in $\mathbb{R}^d$, where 
a graph $G=(V,E)$ is called vertex-redundantly rigid if $G-v$ is rigid for any $v\in V$.

\end{abstract}

\section{Introduction}
A $d$-dimensional {\em bar-joint framework} (or simply a framework) is a pair $(G,\bp)$
of a graph $G=(V,E)$ and a map $\bp:v\in V\mapsto p_v\in \mathbb{R}^d$. 
A pair $(G,\bp)$ is also called a {\em realization} of $G$ in $\mathbb{R}^d$.
We say that  $(G,\bp)$ is {\em equivalent to} $(G,\bq)$, denoted $(G,\bp)\simeq (G,\bq)$, if 
\begin{equation}
\|p_u-p_v\|= \|q_u-q_v\| \qquad \text{for all } uv\in E, \label{eq:bar-length}
\end{equation}
and  we say that $\bp$ and $\bq$ are {\em congruent}, denoted $\bp\equiv \bq$, if
\begin{equation*}
\|p_u-p_v\|=\|q_u-q_v\| \qquad \text{for all pairs } u, v\in V.
\end{equation*}

The set of all maps $\bp$ of the form $\bp:V\rightarrow \mathbb{R}^d$ is denoted by $\mathbb{R}^{dV}$, and 
we will regard $\mathbb{R}^{dV}$ as the $d|V|$-dimensional Euclidean space.
A framework $(G,\bp)$ is called {\em rigid} 
if there is an open neighborhood $U$ of $\bp$ in $\mathbb{R}^{dV}$
such that $\bp\equiv \bq$ holds 
for any $\bq\in U$ with $(G,\bp)\simeq (G,\bq)$.
$(G,\bp)$ is called {\em globally rigid} if $\bp\equiv \bq$ holds 
for any $\bq\in \mathbb{R}^{dV}$ with $(G,\bp)\simeq (G,\bq)$.

A map $\bp\in \mathbb{R}^{dV}$ is called {\em generic} if the set of entries in $\bp$ is algebraically 
independent over $\mathbb{Q}$.
A framework $(G,\bp)$ is called {\em generic} if $\bp$ is generic.
As observed by Gluck~\cite{gluck} and Asimov and Roth~\cite{asimov1978}, 
a generic realization of $G$ is rigid in $\mathbb{R}^d$ 
if and only if every generic realization of $G$ is rigid in $\mathbb{R}^d$.
Hence the generic rigidity can be considered as a property of the underlying graph.
We say that a graph is {\em rigid} in $\mathbb{R}^d$ if 
some/any generic realization of $G$ is rigid in $\mathbb{R}^d$~\cite{asimov1978}.
Laman's theorem~\cite{laman:Rigidity:1970} gives a combinatorial characterization of rigid graphs 
in $\mathbb{R}^2$, 
which states that $G$ is rigid in $\mathbb{R}^2$ if and only if $G$ contains a subgraph $H$
satisfying $|E(H)|=2|V(G)|-3$ and $|E(H')|\leq 2|V(H')|-3$ for any subgraph $H'$ of $H$ with $|V(H')|\geq 2$.
Extending Laman's theorem to rigidity in $\mathbb{R}^3$ is  one of the most important open problems in this field.
See, e.g., \cite{Whiteley:1997} for more details.

The analysis of global rigidity turns out to be  more challenging,
but substantial progress has been made in the last decade. 
The main tool for analyzing global rigidity  is a stress matrix introduced by Connelly~\cite{connelly1982}.
(For the definition of stress matrices, see, e.g.,~\cite{connelly2005}.)
Connelly~\cite{connelly2005} showed that, 
if $(G,\bp)$ is generic and the rank of a stress matrix of $(G,\bp)$ is $|V(G)|-d-1$, then 
$(G,\bp)$ is globally rigid. He also conjectured that the rank condition is necessary.
This conjecture was recently confirmed by Gortler, Healy, and Thurston~\cite{gortler2010}.
An important corollary of their result is:
\begin{theorem}[Gortler, Healy, and Thurston~\cite{gortler2010}]
\label{thm:generic}
If there is a globally rigid generic realization of $G$ in $\mathbb{R}^d$,
then every generic realization of $G$ is globally rigid in $\mathbb{R}^d$.
\end{theorem}
Thus generic global rigidity is a property of graphs, and we can say that a graph $G$ is 
{\em globally rigid} in $\mathbb{R}^d$ if some/any generic realization of $G$ is globally rigid.

Another breakthrough in this context is a combinatorial characterization of 
globally rigid graphs in $\mathbb{R}^2$ by Jackson and Jord{\'a}n~\cite{JJ05}.
In~\cite{H92} Hendrickson showed  the following necessary condition for 
global rigidity of graphs.
\begin{theorem}[Hendrickson~\cite{H92}]
\label{thm:hendrickson}
Let $(G,\bp)$ be a generic framework in $\mathbb{R}^d$.
If $(G,\bp)$ is globally rigid, then either $G$ is a complete graph on at most $d+1$ vertices or 
$G$ is $(d+1)$-connected and redundantly rigid,
where $G$ is called redundantly rigid if $G-e$ is rigid for all $e\in E(G)$.
\end{theorem}

A useful method for generating rigid/globally rigid graphs is the so-called {\em Henneberg operation},
that is, an operation that creates a new  graph by adding a new vertex.
For a graph $G=(V,E)$, a {\em 0-extension} in $\mathbb{R}^d$ 
adds a new vertex with $d$ new edges incident to the new vertex,
and a {\em 1-extension} in $\mathbb{R}^d$ removes an existing edge $e$
and adds a new vertex with $d+1$ new incident edges 
so that the new vertex is incident with the both endvertices of $e$.
It is known that both operations preserve the rigidity of graphs in $\mathbb{R}^d$,
which gives a proof of Laman's theorem in case of $d=2$, see~\cite{laman:Rigidity:1970}.
Analogously, Connelly showed that the 1-extension preserves the corank of stress matrices 
and conjectured that for $d=2$ every graph satisfying Hendrickson's condition can be generated  from $K_4$ by a sequence of 1-extensions and edge additions.
\begin{theorem}[Connelly~\cite{connelly2005}, Jackson, Jord{\'a}n and Szabadka~\cite{JJS06}, Szabadka~\cite{szabadka}]\footnote{What Connelly proved in \cite{connelly2005} is that 1-extension preserves the corank of stress matrices, which implies that 1-extension preserves the global rigidity due to the algebraic characterization by Gortler et al.~\cite{gortler2010}. The present form of Theorem~\ref{thm:connelly} for $d=2$ was first proved by Jackson, Jord{\'a}n and Szabadka~\cite{JJS06} before \cite{gortler2010}. Their proof was extended for general dimension in \cite{szabadka}. 
Szabadka's result is actually more general, see the discussion after Lemma~\ref{lem:1extension-extension} for more detail.}
\label{thm:connelly}
A graph obtained from a globally rigid graph in $\mathbb{R}^d$ by a 1-extension is globally rigid in $\mathbb{R}^d$.
\end{theorem}
The conjecture was confirmed by Berg and Jord{\'a}n~\cite{BJ03} for rigidity circuits and later
by Jackson and Jord{\'a}n~\cite{JJ05} for the general case.
\begin{theorem}[Jackson and Jord{\'a}n~\cite{JJ05}]
\label{thm:jj}
A 3-connected redundantly rigid graph in $\mathbb{R}^2$ can be constructed from $K_4$ by a sequence of 1-extensions and edge additions.
\end{theorem}
Combining these results, we have: 
\begin{theorem}[Connelly~\cite{connelly2005}, Jackson and Jord{\'a}n~\cite{JJ05}]
\label{thm:global-comb}
A graph $G$ is globally rigid in $\mathbb{R}^2$ 
if and only if $G$ is a complete graph on at most three vertices or $G$ is 3-connected and redundantly rigid in $\mathbb{R}^2$.
\end{theorem}

Hendrickson~\cite{H92} conjectured the converse of Theorem~\ref{thm:hendrickson},
and Theorem~\ref{thm:global-comb} confirms Hendrickson's conjecture in $d=2$.
On the other hand Connelly pointed out that the conjecture is false for $d\geq 3$.
In this paper we prove that a slightly stronger condition implies the global rigidity.
Namely we show that, if $G$ is {\em vertex-redundantly} rigid in $\mathbb{R}^d$, then $G$ 
is globally rigid in $\mathbb{R}^d$, where $G$ is said to be vertex-redundantly rigid if 
$G-v$ is rigid for all $v\in V(G)$.

%
%Although establishing a higher dimensional couterpart seems difficult due to the lack of knowledge of rigidity,
%it is natural and important to prove combinatorial characterizations of global rigidity of 
%restricted classes of graphs whose higher dimensional  rigidity has been characterized.
%One of important such classes is a family of body-bar graphs.
%A {\em body-bar framework} is a structure  consisting of bodies linked by dijoint bars.
%By regarding each body as a dense bar-joint subframework, one can consider a body-bar framework 
%as a special case of bar-joint framework whose underlying graph is called a body-bar graph.
%(A formal definition is given in Section~?.)
%Tay gave a combinatorial characterization of the rigidity of body-bar graphs in $\mathbb{R}^d$, 
%and recently Connelly, Jord{\'a}n and Whiteley confirmed that 
%Henderickson's condition is necessary and sufficient for the global rigidity of body-bar graphs in $\mathbb{R}^d$
%based on an inductive construction proved by Frank and ? ,

More generally, in this paper, we shall investigate 
how to find globally rigid generic realizations of graphs in $\mathbb{R}^d$
based on elementary geometric observations.
Our arguments lead to new proofs of combinatorial characterizations of 
the global rigidity of graphs in $\mathbb{R}^2$ (Theorem~\ref{thm:jj}) 
and that of body-bar graphs in $\mathbb{R}^d$ recently shown by Connelly, Jord{\'a}n, and Whiteley~\cite{CJW} 
(Theorem~\ref{thm:body-bar}).
We also give extensions of the 1-extension theorem (Theorem~\ref{thm:connelly})
and Connelly's composition theorem~\cite{connelly2011combining}, 
which are main tools for generating globally rigid graphs in $\mathbb{R}^d$,
and prove a conjecture by Frank and Jiang~\cite[Conjecture 29]{frank2011} on the global rigidity of $k$-chains.

Before closing the introduction, we should emphasize that 
Theorems~\ref{thm:connelly},~\ref{thm:jj},~\ref{thm:global-comb} were 
proved before the work of Gortler et al.~(Theorem~\ref{thm:generic}).
In particular, the result by Connelly~\cite{connelly2005}, Jackson and Jord{\'a}n~\cite{JJ05} 
implies not only the existence of a globally rigid generic realization of 
a 3-connected redundantly rigid graph but also Theorem~\ref{thm:generic} for $d=2$.
Later, Jackson, Jord{\'a}n and Szabadka~\cite{JJS06} also gave 
a proof of Theorem~\ref{thm:connelly} for $d=2$ without using stress matrices.
In this sense, their results lead to a proof of Theorem~\ref{thm:generic} 
which relies on deep graph-theoretical observations, which is still important 
as the proof of Theorem~\ref{thm:generic} in \cite{gortler2010} 
mainly consists of  geometric/algebraic observations.

However, toward a combinatorial characterization of globally rigid graphs in $\mathbb{R}^3$,
we may now turn our attention to finding {\em a} globally rigid generic realization of a graph
with the aid of Theorem~\ref{thm:generic}.
This paper shows that there are several simple techniques  to find a generic  globally rigid realization,
which lead to much simpler arguments for proving a sufficiency of the global rigidity of graphs.

%We should remark that finding a globally rigid realization is not sufficient for proving 
%the global rigidity of a graph; only a  globally rigid {\em generic} realization can be a certificate
%of global rigidity, see, e.g.,~\cite{connelly2010cone} on this matter.

The paper is organized as follows.
In Section~2 we give preliminary facts on rigidity.
In Section~3, we shall show how to find a globally rigid generic realization of
a body-bar graph.
In Section~4, we shall discuss generalizations of the 1-extension theorem and the composition theorem and give applications.
In Section~5, we shall discuss another application of a result in Section~4 
to the global rigidity of body-hinge frameworks, 
which are important special cases of bar-joint frameworks which appear in many applications.

\section{Preliminaries}

For a graph $G=(V,E)$, we consider a smooth function $f_G:\mathbb{R}^{dV}\rightarrow \mathbb{R}^E$ defined by
\begin{equation}
 \label{eq:rigidity-map}
f_G(\bq)(uv)=\|q_u-q_v\|^2 \qquad (\bq\in \mathbb{R}^{dV}, uv\in E).
\end{equation}
The function is known as the {\em rigidity map} of $G$.
The Jacobian of $f_G$ at $\bp\in \mathbb{R}^{dV}$ is called the {\em rigidity matrix} of $(G,\bp)$, denoted by 
$R(G,\bp)$.

A vector in the kernel of $R(G,\bp)$ can be considered as an assignment $\bmm:V\rightarrow \mathbb{R}^d$, which is called an {\em infinitesimal motion} of $(G,\bp)$. 
It turns out that, for any $d\times d$ skew-symmetric matrix $S$ and $t\in \mathbb{R}^d$, an assignment $\bmm:V\rightarrow \mathbb{R}^d$ defined by $\bmm(v)=Sp_v+t$ is an infinitesimal motion of $(G,\bp)$.
Such an infinitesimal motion is called {\em trivial}.
%, and we hence have
%\begin{equation}
%\label{eq:rank}
% {\rm rank\ } R(G,\bp)\leq d|V|-{d+1\choose 2}
%\end{equation}
%if $\bp(V)$ affinely spans $\mathbb{R}^d$.
We say that $(G,\bp)$ is {\em infinitesimally rigid} if any infinitesimal motion of $(G,\bp)$ is trivial.
If $\bp$ is generic,  $(G,\bp)$ is rigid if and only if $(G,\bp)$ is infinitesimally rigid~\cite{asimov1978}.

Note that there is a one-to-one correspondence between the set of rows of $R(G,\bp)$ and $E$.
Hence, by using the row independence, one can define a matroid on $E$, which is called 
the {\em rigidity matroid} of $(G,\bp)$.
Since the rank of $R(G,\bp)$ is invariant from $\bp$ as long as $\bp$ is generic, 
the rigidity matroid of $(G,\bp)$ is the same as that of $(G,\bq)$ for any generic $\bq$.
Thus we define the {\em generic rigidity matroid} ${\cal R}_d(G)$ as 
the rigidity matroid of $(G,\bp)$ for any generic $\bp\in \mathbb{R}^{dV}$.
A graph $G=(V,E)$ is rigid if and only if $G$ is a complete graph or 
the rank of ${\cal R}_d(G)$ is equal to $d|V|-{d+1\choose 2}$.

A graph $G=(V,E)$ is called a {\em rigidity circuit} in $\mathbb{R}^d$ if $E$ is a circuit in ${\cal R}_d(G)$.
By Laman's theorem, $G$ is a rigidity circuit in $\mathbb{R}^2$ if and only if 
$|E|=2|V|-2$ and $|E(H)|\leq 2|V(H)|-3$ for any proper subgraph $H$ of $G$.

Let $[\bq]$ be the equivalence class of $\bq$ in $\mathbb{R}^{dV}$ defined by the congruence $\equiv$.
For a framework $(G,\bp)$, we denote by $c(G,\bp)$ the number of distinct classes $[\bq]$ with 
$(G,\bq)\simeq (G,\bp)$.
The following elementary observation is a folklore, which turns out to be a key to the subsequent arguments.
See, e.g.,~\cite{gortler2010, JO12, szabadka} for more detailed discussion.
\begin{proposition}
\label{prop:key}
Let $(G,\bp)$ be a generic realization of a graph $G$.
If $(G,\bp)$ is rigid, then $c(G,\bp)$ is finite. 
\end{proposition}

Throughout the paper we shall use the following notation.
For a point $p\in \mathbb{R}^d$ and $\epsilon\in \mathbb{R}$, 
let $B(p,\epsilon)=\{x\in\mathbb{R}^d\mid \|p-x\|< \epsilon\}$ and 
$\partial B(p,\epsilon)=\{x\in\mathbb{R}^d\mid \|p-x\|=\epsilon\}$.
 
For a graph $G=(V,E)$, let $N_G(v)\subseteq V\setminus \{v\}$ be the set of neighbors of a vertex $v$
and let $d_G(v)=|N_G(v)|$. 
For a finite set $X$, let $K(X)$ be the edge set of the complete graph on $X$.
For disjoint finite sets $X$ and $Y$, let $K(X,Y)$ be the edge set of the complete bipartite graph on the vertex classes $X$ and $Y$.
 
A {\em k-separation} of a graph $G$ is a pair $(G_1,G_2)$ of edge-disjoint subgraphs of $G$ 
each with at least $k+1$ vertices such that $G=G_1\cup G_2$ and $|V(G_2)\cap V(G_2)|=k$.
A graph $G$ is said to be $k$-connected if it has at least $k+1$ vertices and has no $j$-separation
for all $0\leq j\leq k-1$. 
%If $(H_1,H_2)$ is a $k$-separator of $H$, then $V(G_1)\cap V(G_2)$ is called a $k$-separator
For a connected graph $G=(V,E)$, 
a vertex set $X\subseteq V$ with $|X|=k$ is called a {\em k-separator} if 
$G-X$ is disconnected.

\section{Global Rigidity of Body-bar Frameworks}
A body-bar framework is a structure consisting of rigid bodies linked by disjoint bars.
By regarding each body as a dense bar-joint framework, one can consider a body-bar framework 
as a special case of bar-joint framework.
The underlying graph of a body-bar framework is a multigraph obtained by identifying each body with a vertex 
and each bar with an edge.

\begin{figure}[t]
\centering
\begin{minipage}{0.4\textwidth}
\centering
\includegraphics[scale=0.5]{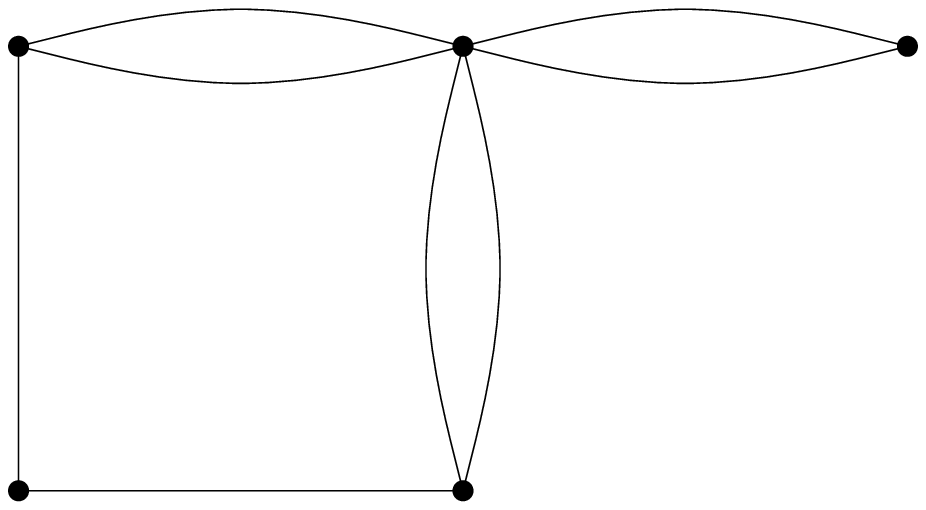}
\par
(a)
\end{minipage}
\begin{minipage}{0.4\textwidth}
\centering
\includegraphics[scale=0.5]{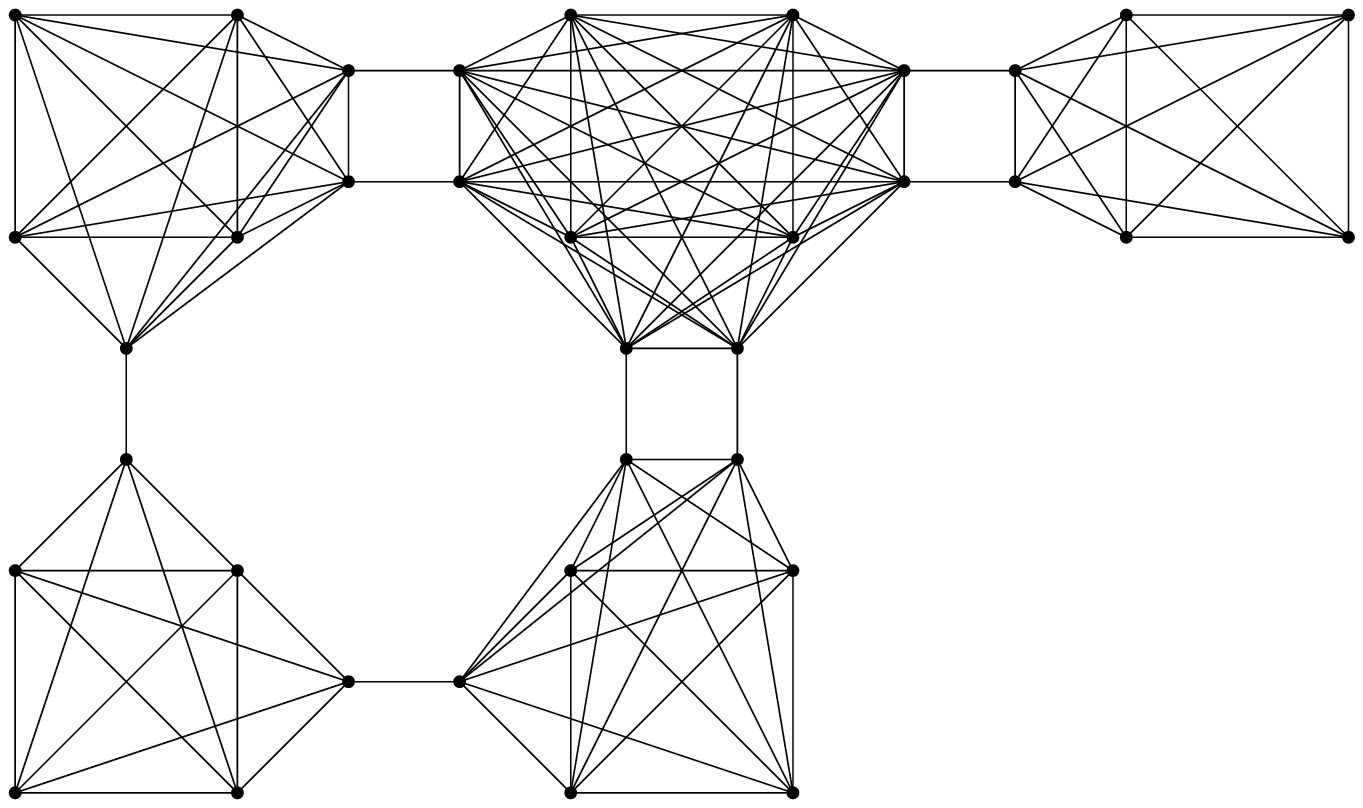}
\par
(b)
\end{minipage}
\caption{A graph $H$ and its body-bar graph $G_H$ for $d=3$.}
\label{fig:body-bar}
\end{figure}

%In a graph $G$, a {\em $j$-clique} is a subgraph in $G$ isomorphic to the complete graph with $j$ vertices.
Given a multigraph $H$ which is the underlying graph of a body-bar framework in $\mathbb{R}^d$,
the associated {\em body-bar graph} is defined by ``replacing'' each vertex $v$ by a complete subgraph
with $(d+1+d_H(v))$ vertices.
Formally, the {\em body-bar graph} of $H$ is a simple graph 
$G_H=(V_H,E_H)$, where 
\begin{itemize}
\item $V_H$ is the union of disjoint vertex sets $B_H^v$ for each $v\in V(H)$
which is defined by $B_H^v=\{v_1^0,\dots, v_{d+1}^0\}\cup \{v_e^1\mid e\in E(H) \text{ is incident to } v\}$;
\item $E_H=(\bigcup_{v\in V} K(B_H^v))\cup \{u_e^1v_e^1\mid e=uv\in E(H)\}$.
\end{itemize}
See Figure~\ref{fig:body-bar} for an example.

Note that $B_H^v$ induces a complete subgraph in $G_H$, which is called a body associated with $v$.
Each body forms a globally rigid subframework when realizing $G_H$ generically.
A bar-joint framework $(G_H,\bp)$ with  $\bp:V_H\rightarrow \mathbb{R}^d$ is called 
a  {\em body-bar realization} of $H$ in $\mathbb{R}^d$.
A celebrated theorem by Tay gives a combinatorial characterization of generically rigid realizability.
\begin{theorem}[Tay~\cite{tay:84}]
\label{thm:tay}
Let $H$ be a finite multigraph.
Then there is a rigid and generic body-bar realization of $H$ in $\mathbb{R}^d$ if and only if 
$H$ contains ${d+1\choose 2}$ edge-disjoint spanning trees.
\end{theorem}

Connelly, Jord{\'a}n and Whiteley~\cite{CJW} recently proved the global rigidity counterpart of Tay's theorem,
which confirms that Hendrickson's conjecture is true in this model. 
\begin{theorem}[Connelly, Jord{\'a}n and Whiteley~\cite{CJW}]
\label{thm:body-bar}
Let $H$ be a multigraph.
Then there is a globally rigid and generic body-bar realization of $H$ in $\mathbb{R}^d$ if and only if
there is a rigid and generic body-bar realization of $H-e$ in $\mathbb{R}^d$ for every $e\in E(H)$.
\end{theorem}
\begin{proof}
The necessity follows from Hendrickson's theorem (Theorem~\ref{thm:hendrickson}) 
and our target is to prove the sufficiency.
The proof is done by induction on $|V(H)|$.
The statement is clear if $|V(H)|=1$, so assume $|V(H)|>1$.

Let us take any $e=uv\in E(H)$ and let $H'=H-e$.
%If there is a globally rigid generic body-bar realization of $H'$, then we are done.
%So let us assume that there is no globally rigid generic realization of $H'$.
Since $G_{H'}$ is rigid, by Proposition~\ref{prop:key} 
there is a generic and rigid body-bar realization $(G_{H'},\bp')$ such that 
$c(G_{H'},\bp')$ is finite.
Let us consider the set of equivalent realizations, that is,
\[
{\cal E}(G_{H'},\bp')=\{\bq':V_{H'}\rightarrow \mathbb{R}^d \mid (G_{H'},\bq')\simeq (G_{H'},\bp'), 
\forall w\in B_{H'}^u, \bq'(w)=\bp'(w)\}
\]
(where we restrict our attention to realizations fixing the body of $u$).
Then ${\cal E}(G_{H'},\bp')$ is a finite set.
Note that the body of $v$ is a complete subgraph, and hence 
for each $\bq'\in {\cal E}(G_{H'},\bp')$ 
there is an isometry
which can be written as a pair of a $d\times d$ orthogonal matrix $A_{\bq'}$
and $t_{\bq'}\in \mathbb{R}^d$ such that $\bq'(w)=A_{\bq'}\bp'(w)+t_{\bq'}$ for all $w\in B_{H'}^v$.

We now extend $\bp':V_{H'}\rightarrow \mathbb{R}^d$ to $\bp:V_H\rightarrow \mathbb{R}^d$ as follows.
Recall that $V_H\setminus V_{H'}=\{u_e^1, v_e^1\}$.
We take $\bp(u_e^1)$ and $\bp(v_e^1)$ in such a way that 
$\{\bp(u_e^1),\bp(v_e^1)\}$ is algebraically independent over 
the field generated by $\mathbb{Q}$ and the entries of $A_{\bq'}$ and $t_{\bq'}$ 
for all $\bq'\in {\cal E}(G_{H'},\bp')$.
Since ${\cal E}(G_{H'},\bp')$ is finite, such $\bp$ exists.

We prove that $(G_{H},\bp)$ is globally rigid in $\mathbb{R}^d$.
To see this let us consider 
\[
{\cal E}(G_{H},\bp)=\{\bq:V_{H}\rightarrow \mathbb{R}^d \mid (G_{H},\bq)\simeq (G_{H},\bp), 
\forall w\in B_{H}^u, \bq(w)=\bp(w)\}.
\]

Since $(G_{H'},\bp')$ is a subframework of $(G_H,\bp)$, 
for each $\bq\in {\cal E}(G_H,\bp)$ there is some $\bq'\in {\cal E}(G_{H'},\bp')$
such that $\bq(w)=A_{\bq'}\bp(w)+t_{\bq'}$ for all $w\in B_H^v$.
In particular, $\bq(v_e^1)=A_{\bq'}\bp(v_e^1)+t_{\bq'}$ holds by $v_e^1\in B_H^v$.
Moreover $\bp(u_e^1)=\bq(u_e^1)$ and $\|\bp(u_e^1)-\bp(v_e^1)\|=\|\bq(u_e^1)-\bq(v_e^1)\|$ 
due to the existence of edge $u_e^1v_e^1$.
Hence, 
\[
 \|\bp(u_e^1)-\bp(v_e^1)\|^2-\|\bp(u_e^1)-A_{\bq'}\bp(v_e^1)-t_{\bq'}\|^2=0.
\] 
If we regard the left hand side as a polynomial with variables $\bp(u_e^1)$ and $\bp(v_e^1)$,
then this polynomial is not identically zero unless $A_{\bq'}$ is the identity matrix and $t_{\bq'}=0$.
(For example, take $\bp(u_e^1)=\bp(v_e^1)=p$ for any $p\in \mathbb{R}^d$ with $(I-A_{\bq'})p\neq t_{\bq'}$.)
Thus, due to the choice of $\bp(u_e^1)$ and $\bp(v_e^1)$, 
$A_{\bq'}$ is the identity and $t_{\bq'}=0$.

Consequently, for any $\bq\in {\cal E}(G_H,\bp)$, 
we have $\bq(w)=\bp(w)$ for all $w\in B_H^u\cup B_H^v$,
and hence the subgraph induced by $B_H^u\cup B_H^v$ can be regarded as one body;
More precisely, there is a globally rigid generic body-bar realization of $H$ if and only if 
there is a globally rigid generic body-bar realization of $H^*$, where 
$H^*$ is the graph obtained from $H$ by contracting $u$ and $v$.
Clearly there is a generic rigid body-bar realization of $H^*-f$ for all $f\in E(H^*)$,
and hence there is a globally rigid generic body-bar realization of $H^*$ by induction.
This completes the proof.  
\end{proof}

\noindent{\em Remark.}
Theorem~\ref{thm:body-bar} along with  Theorem~\ref{thm:generic} and Theorem~\ref{thm:tay} implies 
a combinatorial characterization of 
the global rigidity of body-bar graphs, which is the main theorem of \cite{CJW}.
We should remark that the proof in \cite{CJW} is done by the evaluation of the rank of stress matrices 
by using a constructive characterization of Frank and Szeg{\H o}~\cite{frank}.
Thus their proof further implies Theorem~\ref{thm:generic} for body-bar graphs.

\noindent{\em Remark.}
The definition of body-bar graphs in \cite{CJW} is slightly different from the one given above.
In \cite{CJW} each vertex of $H$ is replaced by a complete graph on $d_H(v)$ vertices.
One can easily check that this distinction does not cause any difference of the global rigidity property.

\section{Global Rigidity of Bar-joint Frameworks}
In this section we shall discuss the global rigidity of generic bar-joint frameworks in $\mathbb{R}^d$.
\subsection{Vertex Removal Lemma}
The following lemma extends Theorem~\ref{thm:connelly}, which turns out to be a very powerful tool as shown in applications in subsequent sections.
\begin{lemma}
\label{lem:1extension-extension}
Let $G=(V,E)$ be a graph and $v$ be a vertex of degree more than $d$. Suppose that 
\begin{itemize}
\item $G-v$ is rigid in $\mathbb{R}^d$, and 
\item $G-v+K(N_G(v))$ is globally rigid in $\mathbb{R}^d$.
\end{itemize}
Then $G$ is globally rigid in $\mathbb{R}^d$.
\end{lemma}
\begin{proof}
Let $(G-v,\bp')$ be a generic realization of $G-v$, and let us take any pair of distinct vertices $i, j$ in $N_G(v)$ such that $ij\notin E(G)$ .
Since $G-v$ is rigid,  we may assume that $c(G-v,\bp')$ is finite by Proposition~\ref{prop:key}.
Hence, the number of possible distances between $q_i'$ and $q_j'$ 
over all $\bq':V(G-v)\rightarrow \mathbb{R}^d$ with $(G-v,\bq')\simeq (G-v,\bp')$ is finite. 
This implies that there is $\epsilon>0$ such that, for any $\bq'$ with $(G-v,\bq')\simeq (G-v,\bp')$,
$|\|p_i'-p_j'\|-\|q_i'-q_j'\||\leq \epsilon$ implies $\|q_i'-q_j'\|=\|p_i'-q_j'\|$.

Let us extend $\bp'$ to $\bp:V(G)\rightarrow \mathbb{R}^d$ such that
$\bp$ is generic and $p_v\in B(p_i,\frac{\epsilon}{4})$.

Consider any $\bq:V(G)\rightarrow \mathbb{R}^d$ with $(G,\bq)\simeq (G,\bp)$.
By $p_v\in B(p_i,\frac{\epsilon}{4})$ and $i, j\in N_G(v)$, we have $|\|p_i-p_j\|-\|q_i-q_j\||\leq \epsilon$. Since $(G-v,\bp')$ is a subframework of $(G,\bp)$, it follows that $\|q_i-q_j\|=\|p_i-p_j\|$.
Therefore, $(G,\bp)$ is globally rigid if $(G+ij, \bp)$ is globally rigid.
In other words $G$ is globally rigid if (and only if) $G+ij$ is globally rigid. 

Applying the same argument to $G+ij$ and then repeatedly to the resulting larger graphs we conclude that $G$ is globally rigid if and only if  
$G+K(N_G(v))$ is globally rigid.
The global rigidity of $G+K(N_G(v))$ now follows from the global rigidity of $G-v+K(N_G(v))$ since $v$ has degree more than $d$.
\end{proof}

%Lemma~\ref{lem:1extension-extension} extends the 1-extension theorem given in Theorem~\ref{thm:connelly}.
%To see this, suppose that $G'$ is a globally rigid graph in $\mathbb{R}^d$ 
%and $G$ is a graph obtained from $G'$ by a 1-extension by adding a new vertex $v$.
%Then $G-v$ is rigid in $\mathbb{R}^d$ and $G'$ is a subgraph of $G-v+K_G(N(v))$,
%which implies the global rigidity of $G-v+K_G(N(v))$. Hence 
%$G$ is globally rigid in $\mathbb{R}^d$ by Lemma\ref{lem:1extension-extension}.

\noindent{\em Remark.} 
Lemma~\ref{lem:1extension-extension} is also implicit in \cite[Lemma 2.3.1]{szabadka} with a completely different argument. It should be noted that Szabadka's proof does not rely on Theorem~\ref{thm:generic}. See also \cite{JJS06, JO12}.

Our proof of Lemma~\ref{lem:1extension-extension} relies on Theorem~\ref{thm:generic}. It is possible to find  {\em a} generic globally rigid realization based on elementary geometric observations without using Theorem~\ref{thm:generic} in $d\leq 3$. Namely one can show the following only using Proposition~\ref{prop:key}.
\begin{theorem}
\label{thm:main3}
Let $G=(V,E)$ be a graph and $v$ be a vertex of degree more than three.
Suppose that 
\begin{itemize}
\item $G-v$ is rigid in $\mathbb{R}^3$, and 
\item there is a generic $\bp':V\setminus \{v\}\rightarrow \mathbb{R}^d$ such that 
$(G-v+K(N_G(v)),\bp')$ is globally rigid in $\mathbb{R}^3$.
\end{itemize}
Then there is a generic $\bp:V\rightarrow \mathbb{R}^d$ extending $\bp'$ such that $(G,\bp)$ is globally rigid in $\mathbb{R}^3$.
\end{theorem}
The proof may be interesting in its own right and is given in the Appendix.

\subsection{Applications}
In this subsection we shall give two applications of Lemma~\ref{lem:1extension-extension}.
\begin{theorem}
\label{thm:vertex-redundant}
Let $G=(V,E)$ be a vertex-redundantly rigid graph in $\mathbb{R}^d$.
Then, $G$ is globally rigid in $\mathbb{R}^d$.
\end{theorem}
\begin{proof}
The proof is done by induction on $|V|$.
When $|V|\leq d+2$, a vertex-redundantly rigid graph $G$ is isomorphic to $K_{|V|}$, 
which is globally rigid in $\mathbb{R}^d$ by definition.

Let us assume $|V|>d+2$. Take any vertex $v\in V$.
If $d_G(v)\leq d$, then the removal of any neighbor of $v$ results in a flexible framework.
Hence, $d_G(v)\geq d+1$ due to the vertex-redundancy of $G$.
Also $G-v$ is rigid.

Let $G'=G-v+K(N_G(v))$.
Suppose that $G'$ is not vertex-redundantly rigid.
Then there is a vertex $u$ such that $G'-u$ is not rigid.
Since $d_G(v)\geq d+1$, $N_G(v)\setminus \{u\}$ induces a complete subgraph of at least $d$ vertices 
in $G'-u$.
Therefore, $G-u$ cannot be rigid, which contradicts the vertex-redundancy of $G$.
Thus $G'$ is vertex-redundantly rigid,
and by induction $G'$ is globally rigid in $\mathbb{R}^d$.
By Lemma~\ref{lem:1extension-extension} $G$ is also globally rigid in $\mathbb{R}^d$. 
\end{proof}

%We remark that, since any rigid graph in $\mathbb{R}^3$ is 3-connected,
%$G$ is 4-connected vertex-redundantly rigid in $\mathbb{R}^3$ if and only if 
%it is vertex-redundantly rigid with minimum degree at least four. 
The extremal properties of vertex-redundantly rigid graphs in $\mathbb{R}^d$ are 
investigated in \cite{servatius,KK}, where vertex-redundantly rigid graphs are referred to as bi-rigid graphs.

We also remark that the characterization for body-bar frameworks (Theorem~\ref{thm:body-bar}) easily follows from Theorem~\ref{thm:vertex-redundant}.

Next we consider the global rigidity of $k$-chains. Following Frank and Jiang~\cite{frank2011}, 
we define  a {\em $k$-chain} as a bipartite graph on disjoint $k$ sets $A_1, A_2, \dots, A_k$ such that $A_i\cup A_{i+1}$ induces the complete bipartite graph on $A_i$ and $A_{i+1}$ for all $1\leq i\leq k-1$ and there are no other edges.
Note that if $k=2$ or $k=3$ a $k$-chain is a complete bipartite graph.
Frank and Jiang~\cite{frank2011} gave a characterization of the global rigidity of $k$-chains with ${d+1\choose 2}$ vertices and conjectured the case of more than ${d+1\choose 2}$ vertices~\cite[Conjecture 29]{frank2011}. 
Here we give an affirmative answer to their conjecture.
\begin{theorem}
Any $(d+1)$-connected $k$-chain with more than ${d+1\choose 2}$ vertices is globally rigid 
in $\mathbb{R}^d$.
\end{theorem}
\begin{proof}
By Bolker and Roth~\cite{bolker}, 
it is known that any $(d+1)$-connected complete bipartite graph is rigid in $\mathbb{R}^d$ 
if the number of vertices is at least ${d+1\choose 2}$. 
Frank and Jiang showed that any $(d+1)$-connected $k$-chain is rigid if $k\geq 4$ and the number of vertices is ${d+1\choose 2}$~\cite[Lemma 17]{frank2011}.
We claim  that  any $(d+1)$-connected $k$-chain $G$ with at least ${d+1\choose 2}$ vertices is rigid.
This can be checked as follows.
If $k\leq 3$ the claim follows from the result of Bolker and Roth.
If $k>3$, then we can take $X_i\subseteq A_i$ with $|X_i|=d+1$ for $i=2, 3$.
By inductively adding vertices to the induced subgraph by $X_2\cup X_3$, one can find a $(d+1)$-connected $k'$-chain $G'$ in $G$ with ${d+1\choose 2}$ vertices for some $k'\leq k$. This subgraph $G'$ is rigid by the result of Frank and Jiang, and $G$ is rigid since $G$ can be obtained from $G'$ by inductively adding the vertices along with at least $d$ edges.

Now consider any $(d+1)$-connected $k$-chain $G$ and without loss of generality 
assume $|A_1|\leq |A_k|$.
Let us take a vertex $v$ from $A_1$.
Then $G-v$ is still a $(d+1)$-connected $k$-chain and hence is rigid.
Also $G-v+K(N_G(v))$ is globally rigid in $\mathbb{R}^d$, 
since $K(N_G(v))=K(A_2)$ is globally rigid by $|A_2|\geq d+1$ and $G-v+K(N_G(v))$ can be obtained from $K(A_2)$ by inductively adding vertices along with at least $d+1$ edges.

Therefore Lemma~\ref{lem:1extension-extension} implies that $G$ is globally rigid in $\mathbb{R}^d$.  
\end{proof}

\subsection{A Proof of Theorem~\ref{thm:global-comb}}
In this section we shall give a new proof of the sufficiency of Theorem~\ref{thm:global-comb}.

Since Lemma~\ref{lem:1extension-extension} is an extension of 1-extension theorem (Theorem~\ref{thm:connelly}), it might be possible to replace Theorem~\ref{thm:jj} with a shorter combinatorial argument for proving Theorem~\ref{thm:global-comb}.
Here we show that the following weaker version of Theorem~\ref{thm:jj} suffices for proving Theorem~\ref{thm:global-comb}.
\begin{lemma}
\label{lem:vertex3-edge}
Let $G$ be a 3-connected redundantly rigid graph in $\mathbb{R}^2$.
Then $G$ has a vertex of degree three or has an edge $e$ for which $G-e$ is 3-connected and 
redundantly rigid.
\end{lemma}
Before showing Lemma~\ref{lem:vertex3-edge} 
let us give a proof of the sufficiency of Theorem~\ref{thm:global-comb}.

\begin{proof}[Proof of the sufficiency of Theorem~\ref{thm:global-comb}]
The proof is done by  induction on the lexicographical order of $(|V|,|E|)$.
The base case is when $G$ is isomorphic to $K_4$, which is globally rigid.

Let us consider the case when  $|V|>4$.
By Lemma~\ref{lem:vertex3-edge}, $G$ has a vertex $v$ of degree three or has an edge $e$ for which 
$G-e$ is 3-connected and redundantly rigid in $\mathbb{R}^2$.
If $G$ has such an edge $e$, then 
by induction there is a globally rigid generic realization $(G-e,\bp)$,
and $(G,\bp)$ is globally rigid.
If $G$ has a vertex of degree three, then $G-v$ is rigid  and 
$G-v+K(N_G(v))$ is 3-connected and redundantly rigid since 
$G$ is 3-connected and redundantly rigid.
Therefore by Lemma~\ref{lem:1extension-extension} $G$ is globally rigid in $\mathbb{R}^2$.
\end{proof}

The remaining of this subsection is devoted to the proof of Lemma~\ref{lem:vertex3-edge}.
For proving it (without using Theorem~\ref{thm:jj}), 
we still need several observations given  by Berg and Jord{\'a}n~\cite{BJ03} 
and Jackson and Jord{\'a}n~\cite{JJ05} at an early stage of the proof of Theorem~\ref{thm:jj}.
Some of these observations are already nontrivial, 
but reasonably short proofs were given there.

To explain these ingredients we need some terminology. 
Let ${\cal M}$ be a matroid on a finite set $E$.
We define a relation $\sim$ on $E$ such that $e\sim e'$ if there is a circuit of ${\cal M}$ 
that contains $e$ and $e'$.
The relation $\sim$ is known to be an equivalence relation, and 
 an equivalence class with respect to $\sim$ is called an 
{\em connected component} of ${\cal M}$.
A subset $F\subseteq E$ is called {\em connected} if $e\sim e'$ for any $e,e'\in F$, 
and ${\cal M}$ is called {\em connected} if $E$ is connected.

%In case of the graphic matroid of $G$, the connected component decomposition corresponds to 
%the 2-connected component decomposition.
%In a 2-connected graph $G$, a subgraph $G'$ is called an {\em ear} if
%$G'$ is a path along $V(G')=\{v_1,\dots, v_k\}$ with $k\geq 2$,
%$d_G(v_1)\geq 3$, $d_G(v_k)\geq 3$, and $d_G(v_i)=2$ for $2\leq i\leq k-1$.
%It is known that a graph  is 2-connected if and only if it can be constructed from $K_3$ by attaching ears.
%
%It is known that the concept can be naturally extended to matroids,
%which was used as an important tool by Jackson and Jord{\'a}n for proving Theorem~\ref{thm:global-comb}.

A sequence  $C_1, C_2,\dots, C_t$ of circuits in ${\cal M}$ 
is called a {\em partial ear decomposition} of ${\cal M}$ 
if for any $2\leq i\leq t$
\begin{itemize}
\item $C_i\cap D_{i-1}\neq \emptyset$, where $D_i=C_1\cup C_2\cup \dots C_i$;
\item $C_i\setminus D_{i-1}\neq \emptyset$;
\item for any circuit $C$ with $C\cap D_{i-1}\neq \emptyset$ and $C\setminus D_{i-1}\neq \emptyset$,
$C\setminus D_{i-1}\subset C_i\setminus D_{i-1}$ does not hold.
\end{itemize}
A partial ear-decomposition is called an {\em ear decomposition} if $\bigcup_{i=1}^t C_i=E$.

\begin{theorem}[Coullard and Hellerstein~\cite{CH96}] 
\label{thm:ear}
A matroid ${\cal M}$ is connected if and only if ${\cal M}$ has an ear decomposition.
If ${\cal M}$ is connected, then any partial ear decomposition can be extended to an ear decomposition. 
\end{theorem}

Now let us come back to rigidity.
Following \cite{JJ05}, 
we say that a graph $G$ is  {\em $M$-connected} in $\mathbb{R}^d$ if ${\cal R}_d(G)$ is connected.
\begin{theorem}[\cite{JJ05}, Theorem 3.2]
\label{thm:3-connected_M_connected}
Suppose that a graph $G$ is 3-connected.
Then $G$ is redundantly rigid in $\mathbb{R}^2$ if and only if $G$ is $M$-connected 
in $\mathbb{R}^2$.
\end{theorem}
\label{thm:redundant-connected}
The following lemma is implicit in \cite[Lemma 5.2]{JJ05} and explicit in \cite{J}.
\begin{lemma}
\label{lem:vertex3}
Let $G$ be an $M$-connected graph in $\mathbb{R}^2$, $C_1, C_2,\dots, C_t$ be an ear decomposition of  
${\cal R}_2(G)$,
and $D_{i}=\bigcup_{j=1}^i C_j$ where $D_0=\emptyset$.
If $V(C_t)\subseteq V(D_{t-1})$, then $|C_t|=1$;
otherwise $G$ has a vertex of degree three that is contained in $V(C_t)\setminus V(D_{t-1})$.
\end{lemma}

We also need certain properties of rigid graphs related to 2-separators.
We say that two 2-separators $\{u,v\}$ and $\{s,t\}$ in $G$ {\em cross}
if $s$ and $t$ are in distinct connected components of $G-u-v$.
\begin{lemma}[\cite{JJ05}, Lemma 3.6]
\label{lem:sep-cross}
If $G$ is rigid in $\mathbb{R}^2$, then any two 2-separators do not cross.
\end{lemma}

Let $G_1$ and $G_2$ be graphs with $V(G_1)\cap V(G_2)=\emptyset$. 
For edge $a_1b_1\in E(G_1)$ and $a_2b_2\in E(G_2)$, the {\em 2-sum $G_1\oplus_2 G_2$ of $G_1$ and $G_2$} (along $a_1b_1$ and $a_2b_2$)
is the graph obtained from $G_1-a_1b_1$ and $G_2-a_2b_2$ by identifying $a_1$ with $a_2$ and $b_1$ with $b_2$.

Conversely, suppose that $(G_1,G_2)$ is a 2-separation of a graph $G$ 
with $V(G_1)\cap V(G_2)=\{a,b\}$.
For $i=1,2$, let $G_i'=G_i+ab$ if $ab\notin E(G_i)$ and otherwise $G_i'=G_i$.
We say that $G_1'$ and $G_2'$ are the {\em cleavage graphs} obtained by {\em cleaving $G$ along $\{a,b\}$}.

\begin{lemma}[\cite{JJ05}, Lemma 3.3]
\label{lem:2sum}
Suppose $G_1$ and $G_2$ are $M$-connected in $\mathbb{R}^2$.
Then the 2-sum $G_1\oplus_2 G_2$ is $M$-connected in $\mathbb{R}^2$.
\end{lemma}

\begin{lemma}[\cite{JJ05}, Lemma 3.4]
\label{lem:cleaving}
Suppose $G_1$ and $G_2$ are obtained from a $M$-connected graph $G$ in $\mathbb{R}^2$ 
by cleaving $G$ along a 2-separator.
Then $G_1$ and $G_2$ are $M$-connected in $\mathbb{R}^2$.
\end{lemma}

For a 2-connected graph $G=(V,E)$, 
a nonempty set $X\subset V$ is called a {\em fragment} if $|N_G(X)|=2$ and $V\setminus (X\cup N_G(X))\neq \emptyset$.
Now we are ready to prove Lemma~\ref{lem:vertex3-edge}.
\begin{proof}[Proof of Lemma~\ref{lem:vertex3-edge}]
By Theorem~\ref{thm:3-connected_M_connected} $G$ is $M$-connected.
Suppose that $G$ has no vertex of degree three.
By Lemma~\ref{lem:vertex3}, for any ear decomposition $C_1,\dots, C_t$ of $G$,
$V(C_t)\subseteq V(D_{t-1})$, where $D_{t-1}=\bigcup_{i=1}^{t-1}C_i$.
Hence, $|C_t|=1$ by Lemma~\ref{lem:vertex3}, and $G-e$ is $M$-connected for $e\in C_t$.
%(in fact, $|C_t|=1$ holds as pointed out in Lemma~\).

Let $E'=\{e\in E(G)\mid G-e \text{ is $M$-connected}\}$,
which is nonempty as we have just noted.
If $G-e$ is 3-connected for some $e\in E'$, then we are done.
Hence we assume that $G-e$ is not 3-connected for any $e\in E'$.

Let us take $e\in E'$ such that the size of the smallest fragment in $G-e$ is minimized 
over all $e\in E'$.
Let $(H_1,H_2)$ be the 2-separator of $G-e$ such that $V(H_1)\setminus V(H_2)$ is the smallest fragment,
and denote $\{a,b\}=V(H_1)\cap V(H_2)$.
By Lemma~\ref{lem:cleaving}, the graphs $H_1', H_2'$ obtained by cleaving along $ab$ 
are $M$-connected.

Let $v$ be the endvertex of $e$  in $H_1'$.
Let us take a circuit $C$ of ${\cal R}_2(H_1')$ 
such that $ab\in C$ and $v\in V(C)$.
Such a circuit exists since $H_1'$ is $M$-connected.
By Theorem~\ref{thm:ear}, there is an ear decomposition $C_1',\dots, C_s'$ of $H_1'$
with $C_1'=C$ for some $s\geq 1$.

If $s=1$, $H'_1$ is a circuit and has a vertex of degree three incident to neither $e$ nor $ab$,
implying that $G$ has a vertex of degree three, a contradiction.
If $s>1$, by Lemma~\ref{lem:vertex3} $H_1'$ has a vertex $u$ of degree three with $u\in V(C_s')\setminus V(C_1')$ or an edge $f\in V(C_s')$ such that $H_1'-f$ is $M$-connected.
In the former case  $u$ is incident to neither $e$ nor $ab$, implying that 
$G$ has a vertex of degree three, a contradiction.
In the later case, since $H_1'-f$ is $M$-connected, $G-f$ is $M$-connected by Lemma~\ref{lem:2sum}.
Hence $G-f$ is not 3-connected, and there is a separator in $G-f$.
However, since $f$ connects two vertices in $H_1$,
Lemma~\ref{lem:sep-cross} implies that a fragment in $G-f$ is properly contained in $V(H_1)\setminus V(H_2)$. This contradicts the choice of $e$.
\end{proof}

\subsection{Combining Two Graphs}
We next consider another operation for constructing globally rigid graphs from smaller graphs.
Let $G$ be a graph, $X$ be a subset of $V(G)$, and $H$ be a graph on $X$ (whose edges may not be in $G$).
We say that $(H,X)$ is a {\em rooted minor} of $(G,X)$ if 
$H$ can be obtained from $G$ by deleting and contracting edges of $G$,
i.e., there is a partition $\{U_v\mid v\in X\}$ of $V(G)$ into $|X|$ subsets, 
each of which is indexed by an element of $X$, such that 
(i) $v\in U_v$, (ii) the induced subgraph of $G$ by $U_v$ is connected for each $v\in X$,
and (iii) $G$ has an edge between a vertex in $U_u$ and a vertex in $U_v$ 
for each $uv\in E(H)$.

\begin{theorem}
\label{thm:comb}
Let $G_1$ and $G_2$ be graphs with $X=V(G_1)\cap V(G_2)$
and $H$ be a graph on $X$ whose edges may not be in $G_1\cup G_2$.
Suppose that 
\begin{itemize}
\item $|X|\geq d+1$,
\item $G_1$ is rigid in $\mathbb{R}^d$,
\item $(H,X)$ is a rooted minor of $(G_2,X)$.
\item $G_1\cup H$ and $G_2\cup K(X)$ are globally rigid or 
$G_1\cup K(X)$ and $G_2\cup H$ are globally rigid in $\mathbb{R}^d$.
\end{itemize}
Then $G_1\cup G_2$ is globally rigid in $\mathbb{R}^d$.
\end{theorem}
\begin{proof}
Let $\bp:V(G_1\cup G_2)\rightarrow \mathbb{R}^d$ be generic,
and $\bp_i$ be the restriction of $\bp$ to $V(G_i)$ for $i=1,2$.
%Also let $v$ be the vertex in $X$ that is incident to the remaining $k-1$ vertices 
%in $K_{1,k-1}$ when taking the rooted minor from $(G_1-F,X)$
Since $(H,X)$ is a rooted minor of $(G_2,X)$, 
there is a partition $\{U_v\mid v\in X\}$ of $V(G_2)$ into $|X|$ subsets, each of which is  indexed by 
an element of $X$,
such that, for each $xy\in E(H)$, there is a path $P_{xy}$ from $x$ to $y$ such that 
$P_{xy}$ is in  the induced subgraph of $G_2$ by $U_x\cup U_y$ 
and passes through an edge between $U_x$ and $U_y$ exactly once. 

We shall assume that, for each $w\in U_v$,  
$p_{w}\in B(p_v,\epsilon)$ holds for some $\epsilon>0$.
Due to the existence of $P_{xy}$ for each $xy\in E(H)$,  
there is a constant $C$ such that, 
for any $\bq_2$ with $(G_2,\bq_2)\simeq (G_2,\bp_2)$,
\begin{equation}
\label{eq:comb}
 |\|p_x-p_y\|-\|q_x-q_y\||\leq C|V(G_2)|\epsilon 
\end{equation}
for all $xy\in E(H)$.

Since $(G_1,\bp_1)$ is rigid, $c(G_1,\bp_1)$ is finite.
 Therefore, for each $xy\in E(H)$, there are finite number of possible distances between 
$q_x$ and $q_y$ over all $\bq$ with  $(G_1,\bq_1)\simeq (G_1,\bp_1)$.
Since the position of $p_w$ for $w\in V(G_2)\setminus X$ is independent from these distances,
if we take small enough $\epsilon$, we have 
\[
 \|p_x-p_y\|=\|q_x-q_y\| \qquad \text{for } xy\in E(H)
\]
for every $\bq$ with $(G_1\cup G_2,\bq)\simeq (G_1\cup G_2,\bp)$.
Therefore, $G_1\cup G_2$ is globally rigid in $\mathbb{R}^d$ 
if $G_1\cup G_2\cup H$ is globally rigid.

By assumption, for some $i\in \{1,2\}$,
$G_i\cup H$ is globally rigid.
However, since $G_1\cup G_2\cup H$ contains $G_i\cup H$ as a subgraph 
and $(G_i\cup H,\bp_i)$ is globally rigid in $\mathbb{R}^d$,
\[
 \|p_v-p_w\|=\|q_u-q_w\| \qquad \text{for all } u,w\in X
\]
for every $\bq$ with $(G_1\cup G_2\cup H,\bq)\simeq (G_1\cup G_2\cup H,\bp)$.
Thus $G_1\cup G_2\cup H$ is globally rigid if $G_1\cup G_2\cup K(X)$ is globally rigid.
Since $|X|\geq d+1$ and $G_1\cup K(X)$ and $G_2\cup K(X)$ are globally rigid,
$G_1\cup G_2\cup H$ is globally rigid.
This in turn implies the global rigidity of $(G,\bp)$. 
\end{proof}

Theorem~\ref{thm:comb} extends \cite[Theorem 11]{connelly2011combining},
which copes with the case when $|X|=d+1$, $|E(H)|=1$, $G_i\cup H$ are globally rigid for each $i=1,2$.
Indeed, in this case, $G_i\cup H$ is $(d+1)$-connected and redundantly rigid by Theorem~\ref{thm:hendrickson}. 
So $G_i$ is rigid. Denoting the edge in $E(H)$ by $ab$,  $G_2-(X\setminus \{a,b\})$ is connected.
Hence $(H,X)$ is a rooted minor of $(G,X)$.

Notice also that Theorem~\ref{thm:comb} contains Theorem~\ref{thm:connelly} as a special case.
In this sense, Theorem~\ref{thm:comb} is a common generalization of 
the 1-extension theorem and the composition theorem \cite[Theorem 11]{connelly2011combining} by Connelly.

\section{Global Rigidity of Body-hinge Frameworks}
In this section we consider body-hinge frameworks.
A body-hinge framework is a structural model consisting of rigid bodies connected by hinges.
A hinge is a $(d-2)$-dimensional affine space which links two bodies, and the bodies can rotate around the hinge.

\begin{figure}[t]
\centering
\begin{minipage}{0.4\textwidth}
\centering
\includegraphics[scale=0.5]{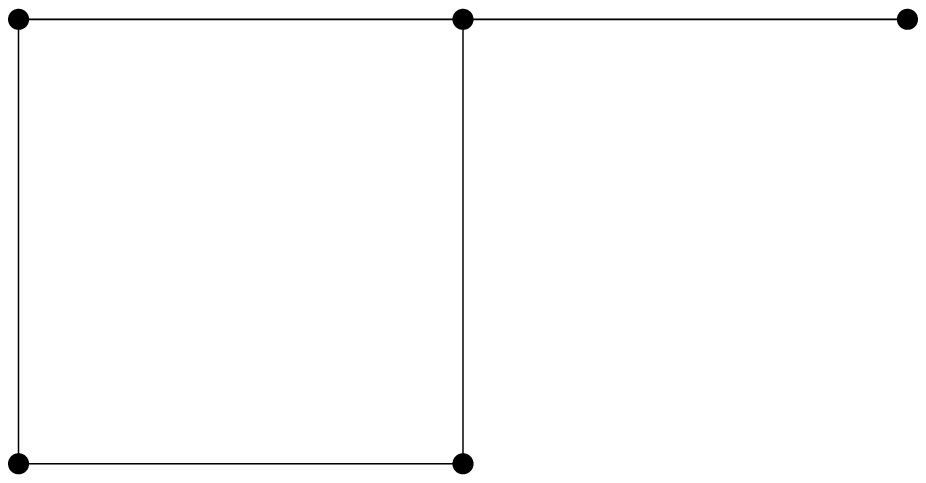}
\par
(a)
\end{minipage}
\begin{minipage}{0.4\textwidth}
\centering
\includegraphics[scale=0.5]{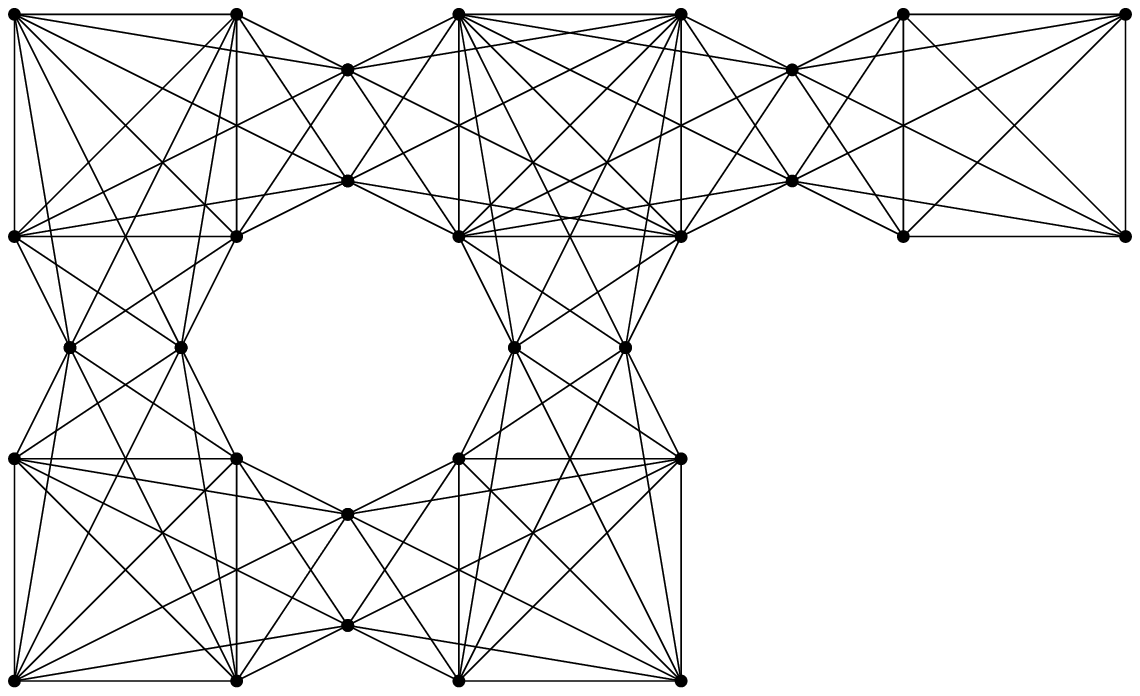}
\par
(b)
\end{minipage}
\caption{(a) $H$ and (b) $G_H$ in $\mathbb{R}^3$.}
\label{fig:body-hinge}
\end{figure}

Given a body-hinge framework, where each hinge connects two bodies, we shall define the underlying graph $H$ by associating a vertex with a body and an edge with each hinge.
As in the case of body-bar frameworks one can regard a body-hinge framework as a bar-joint framework by replacing each rigid body with a complete bar-joint framework. 
Since each hinge is a $(d-2)$-dimensional affine subspace, we are interested in the rigidity of the associated {\em body-hinge graph} $G_H=(V_H,E_H)$ defined as follows:
%For a simple graph $H$ which is the underlying graph of a body-hinge framework in $\mathbb{R}^3$,
%the associated {\em body-hinge graph} $G_H=(V_H,E_H)$ is defined as follows:
\begin{itemize}
\item $V_H=(\bigcup_{v\in V(H)} B_v)\cup (\bigcup_{e\in E(H)} H_e)$, where
$B_v=\{v_1,\dots, v_{d+1}\}$ for each $v\in V(H)$ (which induces the body associated with $v$) 
and $H_e=\{h_{e,1}, h_{e,2},\dots, h_{e,d-1}\}$ for each $e\in E(H)$ (which induces the hinge of $e$);
\item $E_H=(\bigcup_{v\in V(H)} K(B_v))\cup (\bigcup_{e \in E(v)} K(B_v, H_e))$, 
where $E(v)$ denotes the set of edges incident to $v$.
\end{itemize}
Figure~\ref{fig:body-hinge} provides an example.
A bar-joint framework $(G_H,\bp)$ with $\bp:V_H\rightarrow \mathbb{R}^d$ is called a {\em body-hinge realization} 
of $H$ in $\mathbb{R}^d$.
We are interested in the rigidity/global rigidity of $G_H$ for a given $H$.

For a graph $H$ and a positive integer $k$, let $kH$ denote the graph obtained from $H$ by replacing each edge by $k$ parallel copies. For each edge $e\in E(H)$, $ke$ denote the set of the $k$ copies of $e$ in $kH$.  
Tay~\cite{tay:89,tay:91} and Whiteley~\cite{whiteley:88} independently gave a combinatorial characterization of 
the rigidity of body-hinge frameworks.

\begin{theorem}[Tay~\cite{tay:89, tay:91}, Whiteley~\cite{whiteley:88}]
\label{thm:rigidity-body-hinge}
Let $G_H$ be the body-hinge graph of a graph $H$.
Then $G_H$ is rigid  in $\mathbb{R}^d$ if and only if 
$({d+1\choose 2}-1)H$ contains ${d+1\choose 2}$ edge-disjoint spanning trees.
\end{theorem}

Connelly and Whiteley~\cite{connelly2010cone} conjectured that $K_{5,5}$ is the only 
4-connected redundantly rigid graph in $\mathbb{R}^3$ that is not globally rigid.
Observe now that, for a graph $H$,  if $G_H$ is rigid, then $G_H$ is 4-connected and redundantly rigid unless $|V(H)|=1$.
Therefore,  a special case\footnote{This was stated as a conjecture in the draft version (arXiv:1403.3742v1).} of the conjecture by Connelly and Whiteley
states that $G_H$ is rigid in $\mathbb{R}^3$ 
if and only if $G_H$ is globally rigid in $\mathbb{R}^3$.
Recently, counterexamples for this special case were discovered~\cite{JKT}, which also implies the existence of counterexamples for the Connelly-Whiteley conjecture.
 
Also in \cite{JKT} a corrected combinatorial characterization of the global rigidity of body-hinge frameworks were also presented. Here we shall show that a slightly weaker statement follows from Theorem~\ref{thm:vertex-redundant}.
\begin{theorem}
\label{thm:body-hinge}
Let $G_H$ be the body-hinge graph of a graph $H$.
If $({d+1\choose 2}-1)(H-e)+({d+1\choose 2}-3)e$ contains ${d+1\choose 2}$ edge-disjoint spanning trees for every $e\in E(H)$, then $G_H$ is globally rigid in $\mathbb{R}^d$.
\end{theorem}
\begin{proof}
By Theorem~\ref{thm:vertex-redundant} it suffices to show that $G_H$ is vertex-redundantly rigid.
To this end, it suffices to show the rigidity of 
$G_H-h_{e,d-1}$ for each $e\in E(H)$ (where recall that $h_{e,d-1}$ is a vertex in $V_H$ associated with $e$).
Let $G_1=G_H-h_{e,d-1}$.

Intuitively, $(G_1,\bp)$ is a framework 
in which bodies are linked by hinges except that the hinge associated with edge $e$ is replaced with a $(d-3)$-dimensional affine space
(i.e., the bodies incident to the hinge of $e$ are linked at points $p_{h_{e,1}},\dots, p_{h_{e,d-2}}$).
We now show that $G_1$ is rigid. The proof idea is essentially from \cite{whiteley:88, JJT14}.

By the lemma assumption $({d+1\choose 2}-1)(H-e)+({d+1\choose 2}-3)e$ contains ${d+1\choose 2}$ edge-disjoint spanning trees $T_{i,j}, 1\leq i<j\leq d+1$.
Without loss of generality we may assume that 
\begin{equation}
\label{eq:hinge}
\left(\left({d+1\choose 2}-3\right)e\right)\cap \left(T_{d-1,d}\cup T_{d-1,d+1}\cup T_{d,d+1}\right)=\emptyset.
\end{equation}
%where $({d+1\choose 2}-3)e$ denotes the set of the copies of $e$.

Let ${\bm e}_1,\dots, {\bm e}_d$ be the standard basis of $\mathbb{R}^d$, and for simplicity of description the origin of $\mathbb{R}^d$ is denoted by ${\bm e}_{d+1}$.
We shall define $\bp:V(G_H)\rightarrow \mathbb{R}^d$ as follows.
For each $v_i\in B_v$ let $\bp(v_i)={\bm e}_i$.
For each $f\in E(H)\setminus \{e\}$, there is at least one pair $(k,l)$ of indices such that 
$(({d+1\choose 2}-1)f)\cap T_{k,l})=\emptyset$. Hence one can define $\bp(h_{f,i})$ such that 
$\{\bp(h_{f,i})\mid 1\leq i\leq d-1\}=\{{\bm e}_i\mid 1\leq i\leq d+1, i\neq k,l\}$. 
Also, for $e$, we shall define $\bp(h_{e,i})$ such that 
$\{\bp(h_{f,i})\mid 1\leq i\leq d-2\}=\{{\bm e}_i\mid 1\leq i\leq d-2\}$.

Now we show that $(G_1,\bp)$ is infinitesimally rigid in $\mathbb{R}^d$.
To this end let us take any infinitesimal motion $\bmm:V(G_H)\rightarrow \mathbb{R}^d$ of $(G_H,\bp)$. 
The proof is completed if we can show  $\bmm$ is trivial, i.e., 
there is a $d\times d$ skew-symmetric matrix $S$ and 
$t\in \mathbb{R}^d$ such that $\bmm(v)=S\bp(v)+t$ for each $v\in V_H$.
Since $\{\bp(v_i)\mid v_i\in B_v\}$ affinely spans $\mathbb{R}^d$ and $B_v$ induces the complete graph, for each $v\in V(H)$, there is a $d\times d$ skew-symmetric matrix $S_v$ and $t_v\in \mathbb{R}^d$ such that $\bmm(v_i)=S_v\bp(v_i)+t_v$ for every $v_i\in B_v$ and 
 $\bmm(h_{f,j})=S_v\bp(h_{f,j})+t_v$ for every $h_{f,j}$ incident to $B_v$.
 
 For $1\leq i<j\leq d+1$, consider any edge $f=uv\in T_{i,j}$.
 If $f\neq e$, by the definition of $\bp$, there is at least one index $k$ with $1\leq k\leq d-1$ such that 
 $\bp(h_{f,k})$ is either ${\bm e}_i$ or ${\bm e}_j$. 
 Similarly, if $f=e$, there is $k$ with $1\leq k\leq d-2$ such that 
 $\bp(h_{f,k})$ is either ${\bm e}_i$ or ${\bm e}_j$ by (\ref{eq:hinge}).
 Therefore, one can take a vertex $v_l$ from $B_v$ such that 
 $\{\bp(h_{f,k}), \bp(v_l)\}=\{{\bm e}_i, {\bm e}_j\}$.

Let us assume $\bp(h_{f,k})={\bm e}_i$ and $\bp(v_l)={\bm e}_j$.
The first-order length constraint by edge $h_{f,k}v_l$ implies
\begin{align}
\nonumber
0&=\langle \bp(h_{f,k})-\bp(v_l), \bmm(h_{f,k})-\bmm(v_l)\rangle \\ \nonumber
&=\langle \bp(h_{f,k})-\bp(v_l), S_u\bp(h_{f,k})+t_u-S_v\bp(v_l)-t_v\rangle \\ \nonumber
&=\langle {\bm e}_i-{\bm e}_j, S_u{\bm e}_i+t_u-S_v{\bm e}_j-t_v\rangle \\
&=-{\bm e}_j^{\top}S_u{\bm e}_i-{\bm e}_i^{\top}S_v{\bm e}_j+\langle {\bm e}_i-{\bm e}_j, t_u-t_v\rangle. 
\label{eq:hinge2}
\end{align}
 This equation follows even when $\bp(h_{f,k})={\bm e}_j$ and $\bp(v_l)={\bm e}_i$ by changing the role of 
 $u$ and $v$.
 
For $j=d+1$, since ${\bm e}_{d+1}=0$, (\ref{eq:hinge2}) implies 
\begin{equation*}
\langle {\bm e}_i, t_u-t_v\rangle=0 \qquad \text{ for } 1\leq i\leq d \text{ and } uv\in T_{i,d+1}.
\end{equation*}
This implies $t_a=t_b$ for any pair $a, b\in V(H)$ since $T_{i,d+1}$ is a spanning tree. 
Therefore, using the skew-symmetry of $S_v$, (\ref{eq:hinge2}) becomes
\begin{equation*}
{\bm e}_i^{\top}S_v{\bm e}_j={\bm e}_i^{\top}S_u{\bm e}_j \qquad \text{ for } 1\leq i<j\leq d \text{ and } uv \in T_{i,j}.
\end{equation*}
This in turn implies $S_a=S_b$ for any pair $a, b\in V(H)$. 
Thus $\bmm$ is trivial.
\end{proof}

One important corollary of Theorem~\ref{thm:body-hinge} is the following.
\begin{corollary}
\label{cor:body-hinge}
Let $G_H$ be the body-hinge graph of a graph $H$. 
If $G_H$ is {\em hinge-redundantly rigid}
(i.e., $G_{H-e}$ is rigid for any $e\in E(H)$), then $G_H$ is globally rigid.
\end{corollary}
%However the converse implication does not hold. Consider the cycle of length four for an example.

It is also natural to consider a more general body-hinge model where each hinge may connect more than two bodies. Such a framework is called an {\em identified body-hinge framework}.
The underlying graph becomes a hypergraph (or the corresponding bipartite graph),
and a combinatorial characterization of the rigidity is known in terms of a 
count condition of the underlying graphs~\cite{tay:89,tay:91,tan}.
The argument can be easily generalized to show the counterpart of Corollary~\ref{cor:body-hinge} for identified body-hinge frameworks.

A challenging  open problem is to show a characterization for {\em panel-hinge frameworks} 
or {\em hinge-concurrent frameworks},
where in a panel-hinge framework 
the hinges incident to a body lie on  a 2-dimensional affine space,
and in a hinge-concurrent framework 
all the hinges incident to a body intersect at a point.
It is known that in $\mathbb{R}^3$ the infinitesimal rigidity of such frameworks coincides with that of {\em molecular frameworks}, which are bar-joint frameworks whose underlying graphs are the square $G^2$ of graphs $G$.
A combinatorial characterization of the generic rigidity of those frameworks was shown in \cite{KT11}.

\section{Concluding Remarks}
The global rigidity of graphs in complex spaces was analyzed by Jackson and Owen~\cite{JO12} 
and independently by Gortler and Thurston~\cite{gortler2012}.
Bill Jackson pointed out that the proof of Theorem~\ref{thm:body-bar} can be applied even in 
the complex setting.

It is a long standing open problem whether there is a constant connectivity upper bound for the rigidity of graphs in $\mathbb{R}^d$ for $d>2$. Tibor Jord{\'a}n pointed out that the existence of such a bound immediately implies a constant connectivity upper bound for the global rigidity by Theorem~\ref{thm:vertex-redundant}. 

One of the challenging problems in this context  is to determine which 4-connected redundantly rigid graphs are globally rigid in $\mathbb{R}^3$.
Currently it is not clear how large a subclass of 4-connected redundantly rigid graphs we  can construct 
by the operations of Lemma~\ref{lem:1extension-extension} and Theorem~\ref{thm:comb}. 

\section*{Acknowledgment}
We would like to thank Bill Jackson, Tibor Jord{\'a}n, and referees for helpful comments and corrections.
We would also like to thank Tibor Jord{\'a}n for a discussion about applications of Theorem~\ref{thm:vertex-redundant}.

\appendix

\section{Proof of Theorem~\ref{thm:main3}}
\begin{proof}
Let $N_G(v)=\{u_0,u_1,\dots, u_t\}$ and let $\epsilon>0$.
%Take any vertex from $N_G(v)$, and denote it by $u_0$.
%Let $(G-v,\bp')$ be a generic realization of $G-v$ in $\mathbb{R}^d$ and  let $\epsilon>0$.
%and let $B$ be the open ball in $\mathbb{R}^2$ centered at $p_{u_0}'$ with radius $\epsilon$,
We extend $\bp':V\setminus \{v\}\rightarrow \mathbb{R}^d$ to $\bp:V\rightarrow \mathbb{R}^d$ such that
$\bp$ is generic and 
$p_v\in B(p_{u_0},\epsilon)$.
Since $c(G-v,\bp')$ is finite, there is $\epsilon'$ such that, for any $\bq$ with $(G,\bq)\simeq (G,\bp)$, 
$|\|p_{u_i}-p_{u_0}\|-\|q_{u_i}-q_{u_0}\||\leq \epsilon'$ implies $\|p_{u_i}-p_{u_0}\|=\|q_{u_i}-q_{u_0}\|$ for $1\leq i\leq t$. Therefore, if $\epsilon$ is small enough, $(G,\bp)$ is globally rigid if and only if $(G_1,\bp)$ is globally rigid, where $G_1$ is the graph obtained from $G$ by adding $\{u_0u_i\mid u_i\in N_G(v)\setminus \{u_0\} \}$.

%Take any vertex in $N_G(v)\setminus \{u_0\}$ and denote it by $u_i$.
Let $\ell$ be the line through $p_{u_0}'$ and $p_{u_1}'$, 
$r^*$ be a point on $\ell$ between $p_{u_0}$ and $p_{u_1}$ with $r^*\in B(p_{u_0}',\frac{\epsilon}{2})$,
and let $\delta>0$ be a small number which will be specified later.
When extending $\bp'$ to $\bp$, we shall further assume $p_v\in B(r^*,\delta)$. 
We show that $(G_1,\bp)$ (and hence $(G,\bp)$) is globally rigid. 

\begin{claim}
\label{claim:3d-1}
Let  $G_2$ be the graph obtained from $G_1$ by adding $\{u_1u_i\mid u_i\in N_G(v)\setminus \{u_0,u_1\}\}$.
Then $(G_1,\bp)$ is globally rigid in $\mathbb{R}^3$ if and only if $(G_2,\bp)$ is globally rigid in $\mathbb{R}^3$,
\end{claim}
%
%There are a point $q^*\in B(p_{u_0},\epsilon)$ and $\delta>0$ such that,
%if $p_v\in B(q^*,\delta)$, 
%then $\|p_{u_1}-p_{u_i}\|=\|q_{u_1}-q_{u_i}\|$ for every $u_i\in N_G(v)\setminus \{u_0,u_1\}$.
\begin{proof}
Clearly $(G_2,\bp)$ is globally rigid if $(G_1,\bp)$ is globally rigid.
To see the opposite direction, 
let us take any vertex $u_i\in N_G(v)\setminus \{u_0,u_1\}$,
and we now show $\|p_{u_1}-p_{u_i}\|=\|q_{u_1}-q_{u_i}\|$ 
for every $\bq$ with   $(G_1,\bq)\simeq (G_1,\bp)$.

To see this, let $H$ be the plane in $\mathbb{R}^3$ determined by $p_{u_0}, p_{u_1}, p_{u_i}$.
Since $G_1$ contains edges $u_0u_1$ and $u_0u_i$, we may restrict our attention to 
equivalent realizations of the following form:
\begin{equation}
\label{eq:3d-2}
 {\cal E}(G_1,\bp)=\left\{\bq:V\rightarrow \mathbb{R}^3 \Big| \begin{matrix}  
(G_1,\bq)\simeq (G_1,\bp) \\
\ q_{u_0}=p_{u_0}, q_{u_1}=p_{u_1}, q_{u_i}\in H \end{matrix}
\right\}.
\end{equation}
Let us also consider 
\[
 {\cal E}(G_1-v,\bp')=\left\{\bq':V\setminus \{v\}\rightarrow \mathbb{R}^3 \Big| \begin{matrix}  
(G_1-v,\bq')\simeq (G_1-v,\bp') \\
\ q_{u_0}'=p_{u_0}', q_{u_1}'=p_{u_1}', q_{u_i}'\in H \end{matrix}
\right\}.
\]
Since $G_1-v$ is rigid, ${\cal E}(G_1-v,\bp')$ is a finite set.
Thus, $Q_i$, defined by $Q_i=\{q_{u_i}'\mid \bq'\in {\cal R}'(G_1-v,\bp')\}$, is finite.
Moreover, since $G_1-v$ is a subgraph of $G_1$ and $\bp$ is an extension of $\bp'$, 
\begin{equation}
\label{eq:last}
 \{q_{u_i}\mid \bq\in {\cal E}(G_1,\bp)\}\subseteq Q_i.
\end{equation}
Since $G_1-v$ has edge $u_0u_i$, all points of $Q_i$ 
are contained in circle $H\cap \partial B(p_{u_0},\|p_{u_0}-p_{u_i}\|)$.

\begin{figure}[t]
\centering
\begin{minipage}{0.5\textwidth}
\centering
\includegraphics[scale=0.6]{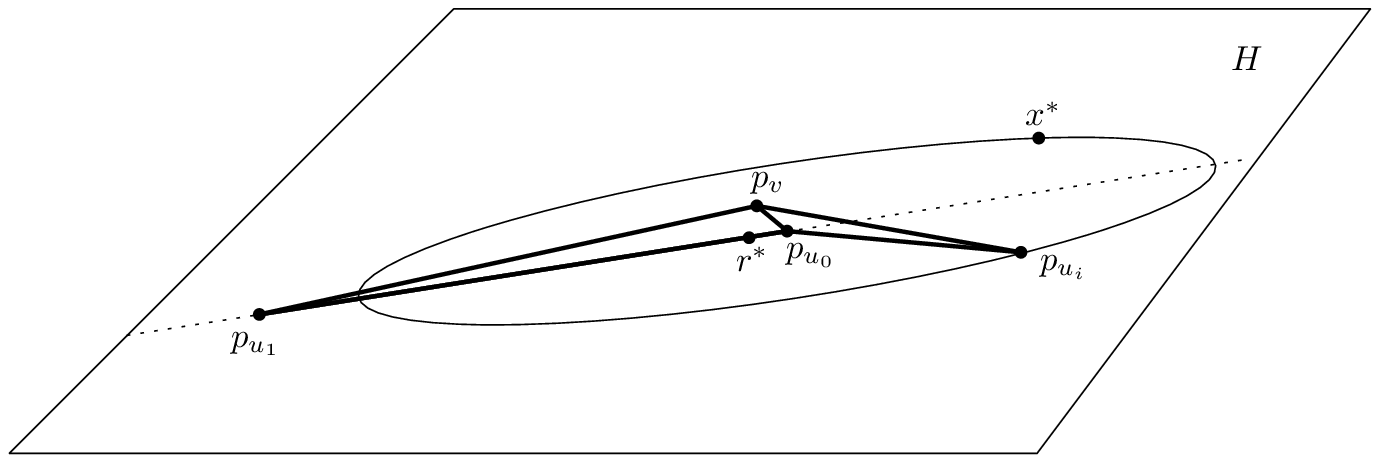}
\par
(a)
\end{minipage}
\begin{minipage}{0.4\textwidth}
\centering
\includegraphics[scale=0.6]{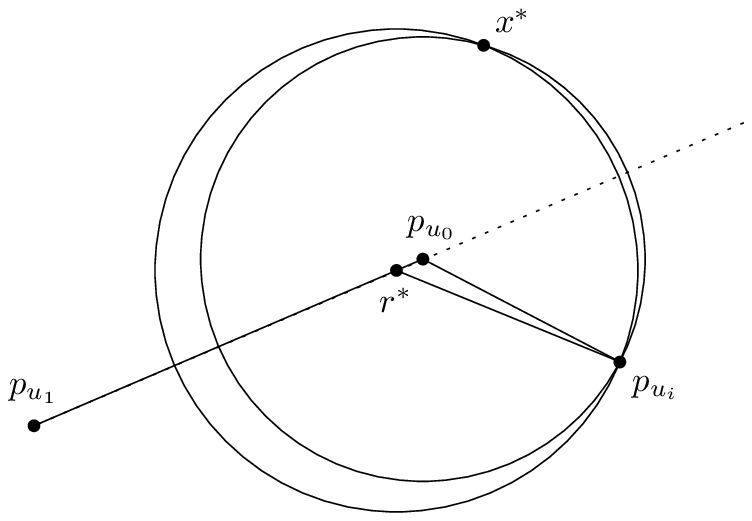}
\par
(b)
\end{minipage}
\caption{Figures in the proof of Theorem~\ref{thm:main3}. 
(a) The plane and the circle illustrate $H$ and $\partial B(p_{u_0},\|p_{u_0}-p_{u_i}\|)$.
All points of $Q_i$ are contains the circle.
(b) The view of figure (a) from the top,
where the two circles illustrate $\partial B(p_{u_0},\|p_{u_0}-p_{u_i}\|)$ 
and $\partial B(r^*,\|r^*-p_{u_i}\|)$.}
\label{fig}
\end{figure}

Recall that $\ell$ denotes the line through $p_{u_0}$ and $p_{u_1}$.
Let $x^*$ be the point obtained by reflecting $p_{u_i}$ along $\ell$ on $H$. 
(See Figure~\ref{fig}.)
Then $H\cap \partial B(r^*,\|r^*-p_{u_i}\|)$ intersects $H\cap \partial B(p_{u_0},\|p_{u_0}-p_{u_i}\|)$ at 
$p_{u_i}$ and $x^*$.

Let $\br:V\rightarrow \mathbb{R}^3$ be an extension of $\bp'$ such that $\br(v)(=r_v)=r^*$.
Since $G_1$ contains edges $\{u_0u_1, u_0v, u_1v\}$, 
\begin{equation}
\label{eq:3d-5}
 r^*=r_v=q_v \text{ for any } \bq\in {\cal E}(G_1,\br)
\end{equation}
(where ${\cal E}(G_1,\br)$ is defined by replacing $\bp$ with $\br$ in (\ref{eq:3d-2})).
Furthermore, since $G_1$ contains edges $\{u_0u_i,vu_i\}$, 
\begin{equation}
\label{eq:3d-6}
 q_{u_i}=r_{u_i}=p_{u_i} \text{  or } q_{u_i}=x^* \text{ for any } \bq\in {\cal E}(G_1,\br).
\end{equation}

Now we consider $(G,\bp)$ with  $p_{v}\in B(r^*,\delta)$ for some small $\delta>0$.
Because $G_1$ contains edges $\{vu_0, vu_1, vu_i, u_0u_1, u_0u_i\}$, we have the following 
for any $\bq\in {\cal E}(G_1,\bp)$:
\begin{description}
\item[(i)] $q_{v}$ is in circle 
$\{y\in \mathbb{R}^3\mid \|p_{u_0}-y\|=\|p_{u_0}-p_v\|, \|p_{u_1}-y\|=\|p_{u_1}-p_v\|\}$;
\item[(ii)] $q_{u_i}$ is one of the two intersection points 
of two circles $H\cap \partial B(p_{u_0},\|p_{u_0}-p_{u_i}\|)$ and $H\cap \partial B(q_{v},\|p_v-p_{u_i}\|)$;
\end{description}
Now we make $\delta\rightarrow 0$.
Then $p_v$ converges to $r^*$ and hence $\bp$ converges to $\br$.
Also, by (i) and (\ref{eq:3d-5}), $q_v$ converges to $r^*$ for any $\bq\in {\cal E}(G_1,\bp)$.
Hence, by (ii) and (\ref{eq:3d-6}), $q_{u_i}$ converges to either $p_{u_i}$ or $x^*$.
This implies that, by taking small enough $\delta>0$,
$q_{u_i}$ can be arbitrary close to either $p_{u_i}$ or $x^*$.
 
Recall that $Q_i$ is a finite set.
Hence there is $\zeta>0$ such that
\[
Q_i\cap (B(p_{u_i},\zeta)\cup B(x^*,\zeta))\subseteq \{p_{u_i},x^*\}. 
\]
Such $\zeta$ is independent from the position of $p_v$ as 
$Q_i$ is determined by $(G_1-v,\bp')$.
However, since (\ref{eq:last}) holds for any choice of $p_v$,
we conclude that 
\begin{equation}
\label{eq:3d-3}
 q_{u_i}=p_{u_i} \text{ or } q_{u_i}=x^* \text{ for all } \bq\in {\cal E}(G_1,\bp) 
\end{equation}
by taking $p_v\in B(r^*,\delta)$ for some small $\delta$,
which in turn implies
\begin{equation}
\label{eq:3d-4}
 \|q_{u_i}-q_{u_1}\|=\|p_{u_i}-p_{u_1}\| \text{ for all } \bq\in {\cal E}(G_1,\bp). 
\end{equation}

Note that (\ref{eq:3d-4}) holds for all $u_i\in N_G(v)\setminus \{u_0,u_1\}$
since the choice of $u_i$ was arbitrary.
\end{proof}

%Let $G_2$ be the graph obtained from $G_1$ by adding 
%$\{u_1u_i\mid u_i\in N_G\setminus \{u_0,u_1\}\}$. 
By Claim~\ref{claim:3d-1}, $(G,\bp)$ is globally rigid if $(G_2,\bp)$ is globally rigid.

%To see that $(G_2,\bp)$ is globally rigid, 
Notice that  $\{u_0,u_1,u_2,v\}$ induces the subgraph isomorphic to $K_4$ in $G_2$.
Hence, to see the global rigidity of $(G_2,\bp)$, we may concentrate on realizations of the form
\[
 {\cal S}'(G_2,\bp)=\left\{\bq:V\rightarrow \mathbb{R}^3 \Big| \begin{matrix}  
(G_2,\bq)\simeq (G_2,\bp) \\
\ q_{u_0}=p_{u_0}, q_{u_1}=p_{u_1}, q_{u_2}=p_{u_2}, q_{v}=p_v \end{matrix}
\right\},
\]
and for analyzing this let us also consider 
\[
 {\cal S}'(G_2-v,\bp')=\left\{\bq':V\setminus \{v\}\rightarrow \mathbb{R}^3\Big| \begin{matrix} 
(G_2-v,\bq')\simeq (G_2-v,\bp') \\
\ q_{u_0}'=p_{u_0}', q_{u_1}'=p_{u_1}', q_{u_2}'=p_{u_2}' \end{matrix}
\right\}.
\]
Since $G_2-v$ is rigid, ${\cal S}'(G_2-v,\bp')$ is finite.

For each $u_i\in N_G(v)\setminus \{u_0,u_1,u_2\}$ and $\bq \in {\cal E}(G_2,\bp)$, 
since $\{v, u_0,u_1,u_i\}$ induces $K_4$ in $G_2$, 
$q_{u_i}$ is equal to either $p_{u_i}$ or the reflection image of $p_{u_i}$
with respect to the mirror plane spanned by $p_{u_0}, p_{u_1}, p_v$.

Since ${\cal S}'(G_2-v,\bp')$ is independent from the position of $p_v$ 
and ${\cal S}'(G_2-v,\bp')$ is finite, 
we can take $p_v\in B(r^*,\delta)$ such that, for every $\bq'\in {\cal S}'(G_2-v,\bp')$,
$q_{u_i}'$ is not the reflection image of $p_{u_i}'$ with respect to the plane spanned by 
$p_{u_0}', p_{u_1}', p_v$.
Since 
\[
 \{q_{u_i}\mid \bq\in {\cal S}'(G_2,\bp)\}\subseteq \{q_{u_i}'\mid \bq\in {\cal S}'(G_2-v,\bp')\},
\]
we conclude that $q_{u_i}=p_{u_i}$ for any $\bq \in {\cal S}'(G_2,\bp)$.
This in turn implies that $(G_2,\bp)$ is globally rigid if 
$(G_3,\bp)$ is globally rigid, where $G_3$ is the graph obtained from $G$ by adding all edges 
between two vertices in $N_G(v)$.
$(G_3,\bp)$ is clearly globally rigid since $(G-v+K(N_G(v)),\bp')$ is globally rigid.
This completes the proof.
\end{proof}

\end{document}